%% file: main.tex
\documentclass{amsart}
\usepackage[utf8]{inputenc}
\usepackage{xcolor}
\usepackage{mathtools}
\usepackage{amsthm}
\usepackage{tikz-cd}
\usepackage{comment}
\usepackage{blindtext}
\usepackage{stix}
\usepackage{natbib}
\usepackage{graphicx}
\usepackage{scalerel}
\usepackage{stackengine}
\usepackage{subcaption}
\stackMath

\numberwithin{equation}{section}

\newtheorem{theorem}{Theorem}[section]
\newtheorem{prop}[theorem]{Proposition}
\newtheorem{lemma}[theorem]{Lemma}
\newtheorem{corollary}[theorem]{Corollary}

\theoremstyle{definition}

\theoremstyle{remark}
\newtheorem*{remark}{Remark}

\newcommand{\R}{\mathbb{R}}
\newcommand{\Z}{\mathbb{Z}}

\newcommand{\ra}{\rightarrow}

\newcommand{\lm}{\Lambda}

\newcommand{\norm}[1]{\left\lVert #1\right\rVert}

\def\vecsign#1{%
  \rule[1.388\LMex]{\dimexpr#1-2.5pt}{.36\LMpt}%
  \kern-6.0\LMpt\mathchar"017E%
}

\def\theraysign#1{%
  \rule{0pt}{17\LMpt}%
  \rule[1.384\LMex]{\dimexpr#1-2.5pt}{.40\LMpt}%
  \kern-6.0\LMpt\mathchar"017E%
}

\title[Variational persistence in the $n$-body problem]
{Variational Persistence of Periodic and Quasi-Periodic Orbits
in Symmetric $N$-Body Problems}
\author{Hangqi Liang, Zhihong Xia, Tingjie Zhou}
\address{Department of Mathematics, Rice University, Houston, TX 77005 USA}
\address{Department of Mathematics, Northwestern University, Evanston, 
  IL 60208 USA}
\address{Chern Institute of Mathematics and LPMC, Nankai University,Tianjin, China}
\email{hl168@rice.edu, xia@math.northwestern.edu, tingjiezhou@nankai.edu.cn}
\date{\today}

\begin{document}

\begin{abstract}
Many mechanical systems, such as the Newtonian $n$-body system, are
invariant under rotations. The rotational symmetry of the system gives
rise to two commuting actions: the dynamical action and the action of
the symmetry group. In this paper, we exploit this symmetry to construct
families of periodic and quasi-periodic trajectories. In the
strong-force setting, we extend Montgomery's variational approach to
free homotopy classes of curves that are periodic up to a prescribed
nontrivial rotation. In the weak-force setting, we establish, under
suitable hypotheses, global and local variational persistence results
for action minimizers in loop spaces defined by general finite
symmetry-group actions under sufficiently small rotations. As applications, we apply these results to the figure-eight orbit, under a
numerically supported assumption of strict local minimality, to double
choreographic loops with $q$-fold rotation symmetry, and to
counter-rotating double choreographic loops in planar $(N+3)$-body problems.
\end{abstract}
\maketitle

\section{Introduction}

This paper concerns periodic and quasi-periodic trajectories of the planar $n$-body problem. We first describe the underlying variational idea in the setting of Lagrangian systems with symmetry; the precise $n$-body formulation is given in Section~2.

Let $M$ be a Riemannian manifold and let $L:TM\to\mathbb R$ be a smooth regular Lagrangian which is strictly convex in the velocity variables. For a fixed $T>0$, the action of a curve $\gamma:[0,T]\to M$ is
\begin{equation}\label{action}
A_T(\gamma)=\int_0^T L(\gamma(t),\dot\gamma(t))\,dt.
\end{equation}
Critical points with fixed endpoints satisfy the Euler--Lagrange equations, whose local flow on $TM$ is denoted by $\phi^t$. Critical points among free loops satisfying $\gamma(0)=\gamma(T)$ extend to $T$-periodic trajectories.

\

We now introduce the concept of symmetry in Lagrangian systems. Let
$G$ be a group. For our application, $G$ is typically $SO(2)$, $SO(3)$
or a permutation group. Let $\alpha: G\times M \ra M$ be a smooth action of
$G$ on $M$. The action $\alpha$ induces an action on the tangent
bundle of $M$. By slightly abusing the notation, we simply denote this induced
action again by $\alpha$.  Also, for any $g \in G$, we write
$\alpha^g: M \ra M$ as the action of element $g$ on $M$. 

We say that $L$ is invariant under action
of $\alpha$, or just $G$ if the action is clear from context, provided that
$$L = L \circ \alpha^g \; \mbox{ for all } \; g\in G.$$

Let $\phi^t: TM \ra TM$ be the Euler-Lagrangian flow generated by the
Lagrangian $L$. The invariance of $L$ under the action of $G$ implies that the symmetry action commutes with the Euler--Lagrange flow:
$$\phi^t\circ\alpha^g=\alpha^g\circ\phi^t,$$
whenever both sides are defined.

Let $$\gamma: [0, T] \ra M$$ be a curve in $M$. We say that $\gamma$
is an {\em equivariant $g$-loop}\/, or simply $g$-loop, if there is an
element $g \in G$ such that
$$\gamma(T) = \alpha^g(\gamma(0)).$$
An equivariant loop is not necessarily a loop in $M$. Of course, if
$g \in G$ is the identity element, then $g$-loop $\gamma$ is a closed
curve in the configuration space $M$.

Given a fixed $g \in G$, let $\gamma^*$ be a critical point of
$A_T(\gamma)$ among all $g$-loops. Then clearly $\gamma^*$ is
a trajectory for the Euler-Lagrange flow. Moreover, its extension by
$$\gamma^*(t) = \alpha^g( \gamma^*(t-T)) , \; \mbox{ for all } \; t
\in \R$$ is a full trajectory. We call such a trajectory a \emph{relative periodic orbit} with relative
period $T$, since it is periodic up to the action of the fixed symmetry
$g$. We also call it an \emph{equivariant $g$-periodic orbit}, or simply
a \emph{$g$-periodic orbit}. A $g$-periodic orbit is not necessarily
periodic. If $\alpha^{g^k}$ is the identity map for some positive integer $k$,
then $g$-periodic orbit of period $T$ is a periodic orbit, but
with period $kT$.

\

For the planar $n$-body problem considered here, the Lagrangian is
$$L(x,\dot x)=K(x,\dot x)+U(x),$$
where $K$ is the kinetic energy and $U$ is the potential function.
We take $G=SO(2)=\R/2\pi\Z$, acting by simultaneous rotations of all
bodies. Both $K$ and $U$ are invariant under this action, and hence
so is $L$. For a rotation through angle $\beta$, the $g$-loop
condition becomes
$$x(T)=e^{i\beta}x(0).$$

If $\beta/(2\pi)=p/q\in\mathbb Q$, with $p\in\mathbb Z$
and $q\in\mathbb N$, then a $\beta$-periodic orbit with
relative period $T$ is periodic with period $qT$.
If $\beta/(2\pi)$ is irrational, excluding the relative
equilibrium case, the orbit is quasi-periodic with two
rationally independent frequencies, and its closure in
phase space is a two-dimensional torus.

\

Our task is therefore to construct collision-free critical points of
the action subject to this boundary condition. We pursue this through
action minimization, where the main difficulties are escape to infinity
and collisions.

Variational constructions of periodic orbits provide the starting point
for our approach. Montgomery \cite{montgomery1998n} established the
existence of action-minimizing reduced periodic solutions in tied free
homotopy classes under a strong-force condition. For weak-force
potentials, symmetry-constrained variational methods have produced a
number of collision-free periodic minimizers. Later in the paper, we apply our results to three such constructions:
the figure-eight solution of Chenciner and Montgomery
\cite{chenciner2000remarkable}, the double choreographic loops with
$q$-fold rotation symmetry of Chen \cite{chen2003binary}, and the
counter-rotating double choreographies in planar $(N+3)$-body problems
of Yuan and Zhang \cite{yuan2018new}.

In the strong-force setting, we extend Montgomery's variational approach
to free homotopy classes of curves that are periodic up to a prescribed
nontrivial rotation.

In the weak-force setting, we establish global and local variational
persistence results under sufficiently small rotations in loop spaces
defined by general finite symmetry-group actions. The resulting
trajectories satisfy the prescribed symmetry constraints when viewed in
the corresponding uniformly rotating frame. Both results require the
type~R compatibility condition (defined in Section~4), which we show to
be sharp for obtaining solutions of the full equations at nonzero
rotation, apart from constant rotating-frame profiles. The global result
additionally assumes a condition equivalent to coercivity of the action
on the unrotated symmetry-constrained loop space.

The global result shows that, if all unrotated global minimizers are
collision-free, then all global minimizers in the corresponding rotated
symmetry classes remain collision-free and approach the unrotated
minimizing set uniformly in $H^1$-distance as the rotation angle tends
to zero. For this global result, we also derive an explicit admissible
rotation bound from a test-path action and a suitable collision-action
lower bound.

The local result gives nearby collision-free local minimizers in the
corresponding rotated symmetry class, converging in $H^1$ to a
prescribed collision-free strict local minimizer in the same limit.

Several previous works obtain rotating or quasi-periodic solutions by
different methods. Chenciner, F{\'e}joz, and Montgomery
\cite{chenciner2005rotating} constructed rotating figure-eight families
by an implicit-function argument under a nondegeneracy assumption,
whereas our local theorem relies on strict local minimality. Ferrario
and Terracini \cite{ferrario2004existence} developed general
collision-exclusion criteria for equivariant minimizers through the
rotating circle property, including in rotating frames. In the planar
three-body problem, Barutello, Ferrario, and Terracini
\cite{Barutello2008Symmetry} obtained more specific collision-exclusion
results for symmetric minimizers in both inertial and uniformly rotating
frames.

Variational constructions of quasi-periodic solutions have also been
obtained in several specific symmetry and boundary classes. Chen
\cite{CHEN_2003} constructed periodic and quasi-periodic action
minimizers for the equal-mass four- and six-body problems; in particular,
his four-body construction includes the cases $d=q=3$ and $d=q=5$
related to our second application. Chen, Ouyang, and Xia
\cite{chen2012action} constructed periodic and quasi-periodic minimizers
under prescribed rotational and reflection symmetries. Liu and Yan
\cite{liu2021exclusion} obtained periodic and quasi-periodic solutions
from free-boundary minimization problems in several planar equal-mass
four- and five-body configurations. These results construct such
solutions directly in particular variational classes, while our
weak-force results give an abstract persistence mechanism from a
prescribed collision-free unrotated minimizer or minimizing set under
sufficiently small rotations.

The paper is organized as follows. Section~2 introduces the variational
setting and the compactness lemma. Sections~3 and~4 establish the
strong-force and weak-force results, respectively. Section~5 presents
applications to the figure-eight orbit, under a numerically supported
assumption of strict local minimality, and to several symmetry-constrained
periodic minimizers, together with numerical illustrations of the
resulting rotating families.

\section{Preliminaries}

\input{part2}

\section{Strong Force}

\input{part3.tex}

\section{Weak Force}

\input{part4.tex}

\section{Applications}

\input{part5.tex}

\bibliographystyle{plain}
\bibliography{cite}

\end{document}

%% file: part2.tex
We begin by introducing the notation and variational framework for the
planar $n$-body problem used throughout the paper. Throughout the paper,
we fix $T>0$. Let $m_1,\ldots,m_n>0$ be the masses of the $n$ bodies and write
$$
x=(x_1,\ldots,x_n)\in(\mathbb R^2)^n,
$$
where $x_i\in\mathbb R^2$ denotes the position of the $i$-th body. For later estimates, set
$$m_{\min}:=\min_{1\leq i\leq n}m_i,\qquad m_{\max}:=\max_{1\leq i\leq n}m_i.$$
We work in the center-of-mass configuration space
$$
V:=
\left\{
x=(x_1,\ldots,x_n)\in(\mathbb R^2)^n:
\sum_{i=1}^n m_i x_i=0
\right\}.
$$

For the planar $n$-body problem, we specialize the symmetry group $G$
introduced in Section~1 to
$$
G=SO(2)\cong\mathbb R/(2\pi\mathbb Z)
\qquad\text{and}\qquad
M=V.
$$
We
define the loop space
$$
\lm:=
\left\{
x\in H^1([0,T],V):x(0)=x(T)
\right\}.
$$

For $\beta\in\mathbb R$, let
$[\beta]\in\mathbb R/(2\pi\mathbb Z)\cong SO(2)$ denote the
corresponding rotation. Identifying $\mathbb R^2$ with $\mathbb C$,
its action on $V$ is given by
$$
\alpha^{[\beta]}x
=
\exp(i\beta)x
=
\left(
\exp(i\beta)x_1,\ldots,\exp(i\beta)x_n
\right).
$$
We say that a curve $y:[0,T]\to V$ is a $\beta$-loop if
$$
y(T)=\exp(i\beta)y(0).
$$
The kinetic energy of $n$ moving masses is given by
$$
K(x,\dot{x})
=
\frac12\sum_{i=1}^n m_i\norm{\dot{x}_i}^2.
$$
We write
$$
V^*:=
\left\{
x=(x_1,\ldots,x_n)\in\mathbb R^{2n}:
\sum_{k=1}^n m_kx_k=0,\;
x_i\neq x_j \text{ for } i\neq j
\right\}
$$
for the collision-free configuration space.

Let $U:(\R^2)^n\to\R_{\geq0}\cup\{\infty\}$ denote the potential
energy. Throughout, we assume that $U$ is lower semi-continuous,
invariant under simultaneous translations and rotations, and is $C^2$ on $V^*$. We also assume that $U(x)=+\infty$ whenever $x_i=x_j$ for some $i\neq j$. Further
assumptions on $U$ will be specified in the corresponding sections. The corresponding Lagrangian is
$$
L(x,\dot{x})=K(x,\dot{x})+U(x).
$$
For $x\in H^1([0,T],V)$, the action \eqref{action} takes the form
$$A_T(x)=\int_0^T L(x(t),\dot x(t))\,dt.$$
Collision-free critical points of the action satisfy
the equations of motion
\begin{equation}\label{equation of motion}
m_i\ddot{x}_i=\frac{\partial U}{\partial x_i}(x),\qquad i=1,\ldots,n.
\end{equation}

One obstruction in the variational problem is that a sequence of curves
with uniformly bounded action may escape to infinity. Once a uniform
$L^\infty$-bound is established, the existence of a limiting curve
follows from a standard compactness argument:

\begin{lemma}\label{lemma:direct method compactness}
Let
$(x^{[k]})_{k\in\mathbb N}\subset
H^1([0,T],V)$ satisfy
$$
\sup_{k\in\mathbb N}
\int_0^T K(x^{[k]}(t),\dot{x}^{[k]}(t))\,dt<\infty,
\qquad
\sup_{k\in\mathbb N}\norm{x^{[k]}}_\infty<\infty.
$$
Then there exist a subsequence, still denoted by
$(x^{[k]})_{k\in\mathbb N}$, and
$x\in H^1([0,T],V)$ such that
$$
x^{[k]}\rightharpoonup x
\quad\text{in }H^1([0,T],V),
\qquad
x^{[k]}\longrightarrow x
\quad\text{in }C^0([0,T],V),
$$
and
$$
A_T(x)\leq\liminf_{k\to\infty}A_T(x^{[k]}).
$$
If, in addition, $x$ is collision-free, $U$ is continuous on $V^*$,
and
$$
A_T(x^{[k]})\longrightarrow A_T(x),
$$
then
$$
x^{[k]}\longrightarrow x
\qquad\text{strongly in }H^1([0,T],V).
$$
\end{lemma}

\begin{proof}

By the Cauchy--Schwarz inequality, for every $0\leq a<b\leq T$,
\begin{equation}\label{eq: equicontinuity estimate}
\begin{split}
\norm{x^{[k]}(a)-x^{[k]}(b)}
&\leq
\int_a^b \norm{\dot{x}^{[k]}(t)}\,dt \\
&\leq
\sqrt{b-a}
\left(
\int_a^b \norm{\dot{x}^{[k]}(t)}^2\,dt
\right)^{1/2} \\
&\leq
\sqrt{b-a}
\left(
\frac{2}{m_{\min}}
\int_a^b K(x^{[k]}(t),\dot{x}^{[k]}(t))\,dt
\right)^{1/2}.
\end{split}
\end{equation}
Since the kinetic parts of the actions are uniformly bounded, the
above estimate shows that $(x^{[k]})_{k\in\mathbb N}$ is uniformly
equicontinuous. The Arzel\`a--Ascoli theorem
therefore gives a subsequence which converges uniformly on $[0,T]$ to
some continuous curve $x$.

The uniform kinetic-energy bound also implies that the derivatives $\dot x^{[k]}$ are uniformly bounded in $L^2([0,T],V)$. Moreover,
$$
\int_0^T\norm{x^{[k]}(t)}^2\,dt
\leq
T\norm{x^{[k]}}_\infty^2,
$$
so $(x^{[k]})_{k\in\mathbb N}$ is uniformly bounded in
$H^1([0,T],V)$.
Therefore, after passing to a
further subsequence, the sequence converges weakly in $H^1$ to some curve
$\widetilde{x}$. The uniform convergence implies convergence in
$L^2$, and hence the weak $H^1$ limit agrees with the uniform limit.
Therefore,
$$
\widetilde{x}=x.
$$
By weak lower semicontinuity of the kinetic term,
$$
\int_0^T K(x(t),\dot{x}(t))\,dt
\leq
\liminf_{k\to\infty}
\int_0^T
K(x^{[k]}(t),\dot{x}^{[k]}(t))\,dt.
$$
Since $x^{[k]}(t)\to x(t)$ for every $t\in[0,T]$ and $U$ is
nonnegative and lower semicontinuous, Fatou's lemma gives
$$
\int_0^T U(x(t))\,dt
\leq
\liminf_{k\to\infty}
\int_0^T U(x^{[k]}(t))\,dt.
$$
Combining these two inequalities, we obtain
$$
A_T(x)
\leq
\liminf_{k\to\infty}A_T(x^{[k]}).
$$
For the last statement, since $x$ is collision-free and $x^{[k]}\to x$ uniformly, continuity
of $U$ on $V^*$ gives
$$
\int_0^T U(x^{[k]}(t))\,dt
\longrightarrow
\int_0^T U(x(t))\,dt.
$$
Using assumed convergence of actions, we have
$$
\int_0^T K(x^{[k]}(t),\dot x^{[k]}(t))\,dt
\longrightarrow
\int_0^T K(x(t),\dot x(t))\,dt.
$$
Together with the weak $H^1$ convergence, this implies
$$
\dot x^{[k]}\longrightarrow\dot x
\qquad\text{strongly in }L^2.
$$
Since $x^{[k]}\to x$ uniformly, we conclude that
$$
x^{[k]}\longrightarrow x
\qquad\text{strongly in }H^1.
$$
\end{proof}
In particular, if $(x^{[k]})_{k\in\mathbb N}$ is a minimizing
sequence in an admissible class which is closed under the above
convergences, then its limit $x$ belongs to that class and satisfies
$$
A_T(x)=\inf A_T.
$$

%% file: part3.tex
Given a potential energy $U:(\mathbb{R}^{2})^n\rightarrow \R_{\geq 0}\cup\{\infty\}$, we say that $U$ satisfies the strong-force law if there exists a constant $c>0$ such that
$$U(x)>\frac{c}{\rho^2(x)}$$
whenever
$$\rho(x):=\min\left\{\norm{x_i-x_j}:1\leq i<j\leq n\right\}$$
is sufficiently small.

Throughout the remainder of this section, we assume that $U$ satisfies
the strong-force law.

For example, the potentials
$$U(x)=\sum_{i<j}\frac{m_im_j}{\norm{x_i-x_j}^2}$$
and
$$U(x)=\sum_{i<j}\frac{m_im_j}{\norm{x_i-x_j}^2}
+\sum_{i<j}\frac{m_im_j}{\norm{x_i-x_j}}$$
both satisfy the strong-force law.
\

By the strong-force condition (see, for example, Montgomery
\cite{montgomery1998n}), any curve $x\in\lm$ containing a collision
satisfies
$$A_T(x)=+\infty.$$
In other words, collision trajectories cannot have finite action for
strong-force potentials.

We now recall the notion of tiedness used by Montgomery\cite{montgomery1998n}. Let $Y$ be a
non-compact complete Riemannian metric space, let $\Sigma\subset Y$,
and let $C$ be a family of curves $x:[0,T]\to Y\setminus\Sigma$. We
say that a sequence $(x^{[k]})_{k\in\mathbb N}\subset C$ escapes to
infinity if there exist $t_k\in[0,T]$ such that
$$d(x^{[k]}(t_k),y_0)\longrightarrow\infty$$
for some fixed base point $y_0\in Y$. The family $C$ is said to be
\emph{tied} to $\Sigma$ if every sequence in $C$ escaping to infinity
satisfies
$$\ell(x^{[k]})\longrightarrow\infty,$$
where $\ell(x^{[k]})$ denotes the length of $x^{[k]}$.

The two possible obstructions to the direct variational method can then
be treated separately: tiedness prevents minimizing sequences from
escaping to infinity, while the strong-force condition excludes
collisions. These two observations form the basis of the following
existence theorem of Montgomery \cite{montgomery1998n}, stated below
for our fixed period $T$.

    Montgomery works on the shape space $\mathcal C:=V/SO(2),$
    and denotes by $\mathcal C^*$ the complement of the collision locus in $\mathcal C$. The kinetic energy induces a Riemannian metric on $\mathcal C$, and the
corresponding action functional on curves in $\mathcal C$ is referred to
as the \emph{reduced action}.

\begin{theorem}[Montgomery]
Suppose that the $n$-body potential satisfies the strong-force
condition and that $\alpha$ is a tied free homotopy class of loops in
$\mathcal C^*$. Then there exists a collision-free
reduced periodic solution of period $T$ representing $\alpha$. The
solution is obtained by minimizing the reduced action over $\alpha$.
\end{theorem}

Whereas Montgomery works with ordinary loops in the reduced
configuration space $\mathcal C$, we work directly in $V$ with
$\beta$-loops, that is, curves satisfying
$$x(T)=\exp(i\beta)x(0).$$
Our result extends his variational existence theorem to free homotopy
classes of $\beta$-loops. No separate tiedness assumption is
needed when $\beta\notin2\pi\mathbb Z$, since the twisted boundary
condition itself implies that every family of $\beta$-loops is tied to
$\varnothing$.

Before stating our result, we introduce the corresponding notion of
homotopy. Two $\beta$-loops $x^{[1]}$ and $x^{[2]}$ are said to be
\emph{freely $\beta$-homotopic} if there exists a continuous map
$$H:[0,1]\times[0,T]\to V^*$$
such that
$$H(0,t)=x^{[1]}(t),\qquad H(1,t)=x^{[2]}(t),$$
and
$$H(s,T)=e^{i\beta}H(s,0)$$
for every $s\in[0,1]$.

\

We can then state our main theorem for strong-force potential.

\begin{theorem}\label{strong force}
For every $\beta\notin2\pi\mathbb Z$, let $\alpha$ be a free
$\beta$-homotopy class of $\beta$-loops in $V^*$. Suppose that the
$n$-body potential satisfies the strong-force condition. Then there
exists a collision-free solution $y^*$ representing $\alpha$. The
solution $y^*$ is obtained by minimizing the action $A_T$ over $\alpha$.

Furthermore, $y^*$ extends to $\mathbb R$ by
$$
y^*(t+kT)=\exp(ik\beta)y^*(t),
\qquad t\in[0,T],\quad k\in\mathbb Z,
$$
and this extension is a $\beta$-periodic solution.
\end{theorem}

\begin{proof}
Fix $\beta\notin2\pi\mathbb Z$ and let $\alpha$ be a free
$\beta$-homotopy class of $\beta$-loops in $V^*$. Let
$(y^{[k]})_{k\in\mathbb N}\subset\alpha$ be an action-minimizing
sequence.

Since $U\geq0$ and the actions $A_T(y^{[k]})$ are uniformly bounded,
the kinetic energies
$$\int_0^T K(y^{[k]}(t),\dot y^{[k]}(t))\,dt$$
are uniformly bounded. By the same Cauchy--Schwarz estimate as in
\eqref{eq: equicontinuity estimate},
$$\ell(y^{[k]})
=
\int_0^T\norm{\dot y^{[k]}(t)}\,dt
\leq
\sqrt{T}
\left(
\frac{2}{m_{\min}}
\int_0^T K(y^{[k]}(t),\dot y^{[k]}(t))\,dt
\right)^{1/2},$$
so the lengths $\ell(y^{[k]})$ are uniformly bounded.

Since each $y^{[k]}$ is a $\beta$-loop,
\begin{equation}
\begin{split}
\ell(y^{[k]})
&\geq
\norm{y^{[k]}(0)-y^{[k]}(T)}\\
&=
\norm{y^{[k]}(0)-\exp(i\beta)y^{[k]}(0)}\\
&=
2\norm{y^{[k]}(0)}
\left|\sin\left(\frac{\beta}{2}\right)\right|.
\end{split}
\end{equation}
Since $\beta\notin2\pi\mathbb Z$, this shows that
$y^{[k]}(0)$ is uniformly bounded. Furthermore, for every
$t\in[0,T]$,
$$\norm{y^{[k]}(t)}
\leq
\norm{y^{[k]}(0)}
+
\norm{y^{[k]}(t)-y^{[k]}(0)}
\leq
\norm{y^{[k]}(0)}+\ell(y^{[k]}).$$
Therefore,
$$\sup_{k\in\mathbb N}\norm{y^{[k]}}_\infty<\infty.$$

Since the hypotheses of Lemma~\ref{lemma:direct method compactness} are
satisfied, after passing to a subsequence, still denoted by
$(y^{[k]})_{k\in\mathbb N}$, we obtain a $\beta$-loop $y^*$ such
that
$$y^{[k]}\longrightarrow y^*$$
uniformly on $[0,T]$, and
$$A_T(y^*)\leq\liminf_{k\to\infty}A_T(y^{[k]}).$$
Once we verify that $y^*\in\alpha$, it follows that $y^*$ attains the
infimum of the action over $\alpha$.

By the strong-force condition, every curve containing a collision has
infinite action. Since $A_T(y^*)<\infty$, the curve $y^*$ is
collision-free.

Since $y^*$ is
collision-free and $[0,T]$ is compact, its image has positive distance
from the collision set. Since $y^{[k]}$ converges uniformly to $y^*$,
for sufficiently large $k$ the homotopy
$$
H(s,t)
=
(1-s)y^{[k]}(t)+sy^*(t),
\qquad
s\in[0,1],
$$
is collision-free. Moreover, since both $y^{[k]}$ and $y^*$ are
$\beta$-loops,
$$
H(s,T)=\exp(i\beta)H(s,0)
$$
for every $s\in[0,1]$. Thus $H$ is a $\beta$-homotopy, and therefore
$y^*\in\alpha$.

We now extend every curve $y\in\alpha$ to $\mathbb R$ by
$$
y(t+kT)=\exp(ik\beta)y(t),
\qquad t\in[0,T],\quad k\in\mathbb Z.
$$
Since the Lagrangian is invariant under rotation, for every
$y\in\alpha$ and every $\epsilon\in\mathbb R$,
$$
\int_\epsilon^{T+\epsilon}
L(y(t),\dot y(t))\,dt
=
A_T(y).
$$
Therefore, $y^*$ remains an action minimizer over $\alpha$ on every
interval $[\epsilon,T+\epsilon]$.

Since $y^*$ is collision-free, $y^*$ is a solution on $(\epsilon,T+\epsilon)$ for every $\epsilon\in\mathbb R$ and 
hence the extension is a $\beta$-periodic solution.
\end{proof}

%% file: part4.tex
\subsection{Symmetry setting and main results}

In this section, we consider the $\alpha$-homogeneous potentials
$$
U(x)=\sum_{1\leq i<j\leq n}
\frac{m_i m_j}{\norm{x_i-x_j}^{\alpha}},
\qquad 0<\alpha<2,
$$
which include the Newtonian potential when $\alpha=1$.
In this range, the strong-force condition is no longer satisfied, and
excluding collisions of action-minimizing trajectories becomes a major
difficulty. We will show how variational existence results for loops can
be extended to the corresponding rotated classes.

Given a family of loops $C\subset\lm$ and a fixed rotation angle
$\beta\in\mathbb R$, we define
$$
C_\beta
:=
\left\{
\exp\left(i\frac{\beta}{T}t\right)x(t):x\in C
\right\}.
$$
We also define
$$\lm_\beta:=\left\{y\in H^1([0,T],V):y(T)=\exp(i\beta)y(0)\right\},$$
and, for $x\in\lm$, the rotated action
$$A_\beta(x):=A_T\left(\exp\left(i\frac{\beta}{T}t\right)x(t)\right).$$
In this way, minimizing $A_T$ over $C_\beta$ is equivalent to minimizing
$A_\beta$ over $C$. We will use these two equivalent formulations as
needed in the proofs.

We are interested in proving the existence of solutions of the
$n$-body equations by minimizing the action over rotated classes
subject to symmetry constraints and possibly additional variational
constraints. Such symmetry constraints can be described through the
equivariant framework of group actions; see
\cite{ferrario2004existence} and the references therein.
We now introduce this symmetry framework.

Let $\mathbb R/T\mathbb Z$
denote the time circle, naturally identified with a circle in
$\mathbb R^2$. Let $H$ be a finite group. Let
$
\tau:H\rightarrow O(2)
$
be an orthogonal representation inducing an action on the time circle,
let
$
\rho:H\rightarrow O(2)
$
be an orthogonal representation on the plane, and let
$
\sigma:H\rightarrow S_n
$
be a group homomorphism satisfying
$$
\sigma(h)(i)=j
\quad\Longrightarrow\quad
m_i=m_j.
$$
The group $H$ then acts on $V$ by
$$
h(x_1,\ldots,x_n)
=
\left(
\rho(h)x_{\sigma(h^{-1})(1)},
\ldots,
\rho(h)x_{\sigma(h^{-1})(n)}
\right).
$$
By identifying every loop $x\in\lm$ with its $T$-periodic extension,
or equivalently with a map on $\mathbb R/T\mathbb Z$, the expression
$$
x(\tau(h)t)
$$
is well-defined. We may therefore define an action of $H$ on $\lm$ by
$$
(h\cdot x)(t)
=
h\left(x\left(\tau(h^{-1})t\right)\right).
$$

In this way, the symmetry constraints are encoded by the action of $H$. The loops satisfying these constraints are precisely the loops fixed by $H$. We denote the corresponding fixed-point spaces by
$$
V^H:=\left\{v\in V:hv=v\text{ for all }h\in H\right\}
$$
and
$$
\lm^H:=\left\{x\in\lm:h\cdot x=x\text{ for all }h\in H\right\}.
$$
We call a subset $D\subset\lm^H$ \emph{$H$-admissible} if:
\begin{enumerate}
    \item $D$ is closed in the weak $H^1$ topology on $\lm^H$.
    \item Every collision-free $x\in D$ is an interior point of $D$ in the strong $H^1$ topology relative to $\lm^H$.
\end{enumerate}

Moreover, following \cite{Barutello2008Symmetry}, we say that $H$ is of \emph{type R} if
$$
\det(\rho(h))\det(\tau(h))=1\qquad\text{for every }h\in H.
$$
Together with $H$-admissibility, the type~R condition allows collision-free constrained minimizers in the rotated classes to yield genuine solutions; its precise role and sharpness are established in Proposition~\ref{prop: collision-free minimizer is solution}.

We now state the main results:

\begin{theorem}\label{theorem: weak force, glocal minimizer version}
Let $H$ be a finite group acting on $\lm$ through $\rho$, $\tau$, and $\sigma$, and let $D\subset\lm^H$ be an $H$-admissible subset containing a curve of finite action. Define
$$\mathcal M_\beta:=\left\{x\in D_\beta:A_T(x)=\inf_{z\in D_\beta}A_T(z)\right\}.$$
Suppose that the following conditions are satisfied:
\begin{enumerate}
\item $V^H=\{0\}$.
\item Every element of $\mathcal M_0$ is collision-free.
\item $H$ is of type~R.
\end{enumerate}
Then there exists $\theta>0$ such that, for every $\beta\in\mathbb R$ with $|\beta|<\theta$, the action $A_T$ attains its minimum on $D_\beta$, every minimizer is collision-free, and each $y_\beta\in\mathcal M_\beta$ extends to a solution on $\mathbb R$ by
$$y_\beta(t+T)=\exp(i\beta)y_\beta(t).$$
Moreover,
$$\sup_{y_\beta\in\mathcal M_\beta}\operatorname{dist}_{H^1}(y_\beta,\mathcal M_0)\longrightarrow0\qquad\text{as }\beta\longrightarrow0.$$
\end{theorem}
Note that the preceding theorem only guarantees that $y_\beta$
approaches the set $\mathcal M_0$, rather than a prescribed minimizer
$x_0\in\mathcal M_0$.

Moreover, a given symmetry group $H$ may fail to be
of type~R, in which case one may need to replace it by a subgroup.
Since this enlarges the corresponding fixed-point space, a solution
which is globally minimizing under the original symmetry constraints
need no longer be globally minimizing under the reduced symmetry group.
These two issues motivate the following local version of the theorem.

\begin{theorem}\label{theorem: weak force, local minimizer version}
Let $H$ be a finite group acting on $\lm$ through
$\rho$, $\tau$, and $\sigma$, and let $D\subset\lm^H$
be an $H$-admissible subset. Suppose that the following conditions
are satisfied:
\begin{enumerate}
\item There exists $x_0\in D$ which is a collision-free strict
$H^1$-local minimizer of $A_T$ over $D$.
\item $H$ is of type~R.
\end{enumerate}
Then there exists $\theta>0$ such that, for every
$\beta\in\mathbb R$ with $|\beta|<\theta$, there exists
$y_\beta\in D_\beta$ which is a collision-free $H^1$-local minimizer
of $A_T$ over $D_\beta$ and extends to a solution on $\mathbb R$
by
$$y_\beta(t+T)=\exp(i\beta)y_\beta(t).$$
Moreover,
$$\norm{y_\beta-x_0}_{H^1}\longrightarrow0\qquad\text{as }\beta\longrightarrow0.$$
\end{theorem}

\subsection{Proofs of the main results}
We now establish the main ingredients needed for the proofs of the two
preceding theorems. As usual, we say that $A_T$ is coercive on a
function space $C$ if a uniform bound on the action over $C$ implies
a uniform bound in the $H^1$-norm.

\begin{prop}\label{prop: uniform coercivity}
Let $H$ be a finite group acting on $\lm$ through
$\rho,\tau,\sigma$. Assume that $\lm^H$ contains a curve
of finite action. The following conditions are equivalent:
\begin{enumerate}
\item[(i)] $V^H=\{0\}$.
\item[(ii)] $A_T$ is coercive on $\lm^H$.
\item[(iii)] $A_T$ is coercive on
$\bigcup\limits_{|\beta|\leq 2\pi-\epsilon}(\lm^H)_\beta$ for any $\epsilon>0$.
\end{enumerate}
\end{prop}

\begin{proof}
The core idea of the proof comes from Proposition~4.1 of
\cite{ferrario2004existence}, which establishes the
equivalence of $(i)$ and $(ii)$. We extend the argument to include
$(iii)$. The implication $(iii)\implies(ii)$ is immediate. It therefore
remains to prove $(ii)\implies(i)$ and $(i)\implies(iii)$.

$(ii)\implies(i)$:
Suppose, for contradiction, that $V^H\neq\{0\}$, and choose
$0\neq v\in V^H$. For any $x\in\lm^H$ with finite action, define
$$
x^{[k]}(t):=x(t)+kv.
$$
Then $(x^{[k]})_{k\in\mathbb N}\subset\lm^H$, while
$$
\norm{x^{[k]}}_{H^1}\to\infty.
$$
On the other hand, for all sufficiently large $k$,
$$
A_T(x^{[k]})\leq A_T(x).
$$
This contradicts the coercivity of $A_T$ on
$\lm^H$. Therefore, $V^H=\{0\}.$

$(i)\implies(iii)$:
If $\epsilon>2\pi$, then the statement is vacuously true. Now fix $\epsilon\in(0,2\pi]$ and let $(y^{[k]})_{k\in\mathbb N}$ be a sequence in
$$
\bigcup_{|\beta|\leq2\pi-\epsilon}(\lm^H)_\beta
$$
whose actions are uniformly bounded by a constant $M>0$. For each
$k$, choose $\beta_k$ with $|\beta_k|\leq2\pi-\epsilon$ and
$x^{[k]}\in\lm^H$ such that
$$
y^{[k]}(t)
=
\exp\left(i\frac{\beta_k}{T}t\right)x^{[k]}(t).
$$

Since $x^{[k]}\in\lm^H$, for every $h\in H$,
$$
\begin{aligned}
h\left(\int_0^T x^{[k]}(t)\,dt\right)
&=
h\left(\int_0^T h^{-1} x^{[k]}(\tau(h)t)\,dt\right)\\
&=
\int_0^T x^{[k]}(\tau(h)t)\,dt\\
&=
\int_0^T x^{[k]}(t)\,dt.
\end{aligned}
$$
Hence
$$
\int_0^T x^{[k]}(t)\,dt\in V^H.
$$
Since $V^H=\{0\}$,
$$
\int_0^T x^{[k]}(t)\,dt=0.
$$
Therefore, by Wirtinger's inequality (see, for example,
Lemma~3.1(a) of \cite{herzog2024minimal}),
\begin{equation}\label{eq:Wirtinger's inequality}
    \norm{x^{[k]}}_{L^2}
\leq
\frac{T}{2\pi}
\norm{\dot x^{[k]}}_{L^2}.
\end{equation}
On the other hand,
$$
\dot y^{[k]}(t)
=
\exp\left(i\frac{\beta_k}{T}t\right)
\left(
\dot x^{[k]}(t)
+
\frac{\beta_k}{T}ix^{[k]}(t)
\right).
$$
Therefore, by \eqref{eq:Wirtinger's inequality},
\begin{equation*}
\begin{split}
\norm{\dot y^{[k]}}_{L^2}
&=
\norm{\dot x^{[k]}+\frac{\beta_k}{T}ix^{[k]}}_{L^2}\\
&\geq
\norm{\dot x^{[k]}}_{L^2}
-
\frac{|\beta_k|}{T}\norm{x^{[k]}}_{L^2}\\
&\geq
\left(1-\frac{|\beta_k|}{2\pi}\right)
\norm{\dot x^{[k]}}_{L^2}\\
&\geq
\frac{\epsilon}{2\pi}\norm{\dot x^{[k]}}_{L^2}.
\end{split}
\end{equation*}
Combining the above estimate with \eqref{eq:Wirtinger's inequality}, we obtain
\begin{equation}\label{eq: estimate L2 norm using L2 of derivative}
\norm{y^{[k]}}_{L^2}
=
\norm{x^{[k]}}_{L^2}
\leq
\frac{T}{2\pi}\norm{\dot x^{[k]}}_{L^2}
\leq
\frac{T}{\epsilon}\norm{\dot y^{[k]}}_{L^2}.
\end{equation}
Finally, since $U\geq0$,
we have
\begin{equation}\label{eq: L2 derivative bounded by action}
    \norm{\dot y^{[k]}}_{L^2}^2
\leq
\frac{2A_T(y^{[k]} )}{m_{\min}}\leq \frac{2M }{m_{\min}}.
\end{equation}
Combining $(\ref{eq: estimate L2 norm using L2 of derivative})$ and $(\ref{eq: L2 derivative bounded by action})$, we have $H^1-$norms of $y^{[k]}$ are uniformly bounded. 
\end{proof}

\begin{prop}\label{prop: estimate of action after applying rotation}
Let $H$ be a finite group acting on $\lm$ through
$\rho,\tau,\sigma$, and suppose that
$$
V^H=\{0\}.
$$
Let $|\beta_1|<2\pi$ and let
$y\in(\lm^H)_{\beta_1}$ have finite action and  satisfy
$$
\int_0^T K(y(t),\dot y(t))\,dt
\leq K_1
$$
for some $K_1>0$. Then, for every $\beta_2\in \R$,
$$
\left|
A_T\left(
\exp\left(i\frac{\beta_2}{T}t\right)y(t)
\right)
-
A_T(y)
\right|
\leq
\delta(\beta_1,\beta_2,K_1),
$$
where
\begin{equation}\label{eq: explicit value of delta}
    \delta(\beta_1,\beta_2,K_1)
:=
\left(
\frac{2|\beta_2|}{2\pi-|\beta_1|}
+
\frac{\beta_2^2}{(2\pi-|\beta_1|)^2}
\right)K_1.
\end{equation}
\end{prop}

\begin{proof}
For each $1\leq i\leq n$,
$$\frac{d}{dt}\left(\exp\left(i\frac{\beta_2}{T}t\right)y_i(t)\right)
=
\exp\left(i\frac{\beta_2}{T}t\right)\left(\dot y_i(t)+i\frac{\beta_2}{T}y_i(t)\right).$$
Since the potential is rotationally invariant,
\begin{equation}\label{eq:estimate of action difference after rotation}
\begin{split}
\left|
A_T\left(\exp\left(i\frac{\beta_2}{T}t\right)y(t)\right)-A_T(y)
\right|
&=
\frac12\left|
\int_0^T\sum_{i=1}^n m_i\left(\left|\dot y_i(t)+i\frac{\beta_2}{T}y_i(t)\right|^2-|\dot y_i(t)|^2\right)dt
\right| \\
&\leq
\frac{|\beta_2|}{T}\int_0^T\sum_{i=1}^n m_i|\dot y_i(t)||y_i(t)|\,dt\\
&\phantom{\leq}\;+
\frac{\beta_2^2}{2T^2}\int_0^T\sum_{i=1}^n m_i|y_i(t)|^2\,dt.
\end{split}
\raisetag{-5\baselineskip}
\end{equation}
We now estimate the two terms separately. Since $y\in(\lm^H)_{\beta_1}$, we may write
$$y(t)=\exp\left(i\frac{\beta_1}{T}t\right)x(t),\qquad x\in\lm^H.$$
Since $V^H=\{0\}$, the argument in the proof of
Proposition~\ref{prop: uniform coercivity} gives
$$\int_0^T x(t)\,dt=0.$$
Therefore, by Wirtinger's inequality,
$$\left(\int_0^T\sum_{i=1}^n m_i|x_i(t)|^2\,dt\right)^{1/2}
\leq
\frac{T}{2\pi}
\left(\int_0^T\sum_{i=1}^n m_i|\dot x_i(t)|^2\,dt\right)^{1/2}.$$
Moreover,
$$\dot x(t)=\exp\left(-i\frac{\beta_1}{T}t\right)\left(\dot y(t)-i\frac{\beta_1}{T}y(t)\right).$$
Combining these with the triangle inequality, we obtain
\begin{equation*}
\begin{split}
\left(\int_0^T\sum_{i=1}^n m_i|y_i(t)|^2\,dt\right)^{1/2}
&=\left(\int_0^T\sum_{i=1}^n m_i|x_i(t)|^2\,dt\right)^{1/2}\\
&\leq\frac{T}{2\pi}\left(\int_0^T\sum_{i=1}^n m_i|\dot x_i(t)|^2\,dt\right)^{1/2}\\
&\leq\frac{T}{2\pi}\left(\int_0^T\sum_{i=1}^n m_i|\dot y_i(t)|^2\,dt\right)^{1/2}\\
&\phantom{\leq}\;+\frac{T}{2\pi}\frac{|\beta_1|}{T}\left(\int_0^T\sum_{i=1}^n m_i|y_i(t)|^2\,dt\right)^{1/2}.
\end{split}
\end{equation*}
Since $|\beta_1|<2\pi$, we may rearrange the preceding inequality to obtain
\begin{equation}\label{eq: L2 bound using twisted kinetic energy}
\int_0^T\sum_{i=1}^n m_i|y_i(t)|^2\,dt
\leq
\frac{T^2}{(2\pi-|\beta_1|)^2}
\int_0^T\sum_{i=1}^n m_i|\dot y_i(t)|^2\,dt
\leq
\frac{2K_1T^2}{(2\pi-|\beta_1|)^2}.
\end{equation}
The
Cauchy--Schwarz inequality and
\eqref{eq: L2 bound using twisted kinetic energy} give
\begin{equation*}
\begin{split}
\int_0^T\sum_{i=1}^n m_i|\dot y_i(t)||y_i(t)|\,dt
&\leq
\left(\int_0^T\sum_{i=1}^n m_i|\dot y_i(t)|^2\,dt\right)^{1/2}
\left(\int_0^T\sum_{i=1}^n m_i|y_i(t)|^2\,dt\right)^{1/2}\\
&\leq
\frac{2TK_1}{2\pi-|\beta_1|}.
\end{split}
\end{equation*}
Substituting the above two estimates into \eqref{eq:estimate of action difference after rotation} gives
\begin{equation}
\begin{split}
\left|
A_T\left(\exp\left(i\frac{\beta_2}{T}t\right)y(t)\right)-A_T(y)
\right|
&\leq
\frac{|\beta_2|}{T}\frac{2TK_1}{2\pi-|\beta_1|}
+
\frac{\beta_2^2}{2T^2}\frac{2K_1T^2}{(2\pi-|\beta_1|)^2}\\
&=
\left(
\frac{2|\beta_2|}{2\pi-|\beta_1|}
+
\frac{\beta_2^2}{(2\pi-|\beta_1|)^2}
\right)K_1.
\end{split}
\end{equation}
Hence we may take
$$\delta(\beta_1,\beta_2,K_1)
:=
\left(
\frac{2|\beta_2|}{2\pi-|\beta_1|}
+
\frac{\beta_2^2}{(2\pi-|\beta_1|)^2}
\right)K_1,$$
which proves the proposition.
\end{proof}

\begin{lemma}\label{lemma: admissible critical point}
Let $H$ be a finite group acting on $\lm$, and let $D\subset\lm^H$ be an $H$-admissible subset. Suppose that $x\in D$ is collision-free, and let $F$ be a $C^1$ functional in a neighborhood of $x$ in $\lm$. If $x$ is a critical point of $F$ restricted to $D$, then $x$ is also a critical point of $F$ restricted to $\lm^H$.
\end{lemma}

\begin{proof}
Since $D$ is $H$-admissible and $x$ is collision-free, $x$ is an interior point of $D$ in the strong $H^1$ topology relative to $\lm^H$. Hence, in a neighborhood of $x$ in $\lm^H$, the sets $D$ and $\lm^H$ coincide. Therefore, criticality of $x$ for $F$ restricted to $D$ implies criticality of $x$ for $F$ restricted to $\lm^H$.
\end{proof}

\begin{prop}\label{prop: collision-free minimizer is solution}
Let $H$ be a finite group acting on $\lm$, and let $D\subset\lm^H$
be an $H$-admissible subset. Let $\beta\neq0$, and let
$$
y_\beta(t)=\exp\left(i\frac{\beta}{T}t\right)x_\beta(t),\qquad x_\beta\in D,
$$
be a collision-free critical point of $A_T$ restricted to $D_\beta$. Then $y_\beta$ satisfies the $n$-body equations on $[0,T]$ if and only if either $H$ is of type~R or $x_\beta$ is constant. In either case, its extension
$$
y_\beta(t+kT)=\exp(ik\beta)y_\beta(t),\qquad t\in[0,T],\ k\in\mathbb Z,
$$
is a solution of the $n$-body equations on $\mathbb R$.
\end{prop}

\begin{proof}
Recall that
$$A_\beta(x)=A_T\left(\exp\left(i\frac{\beta}{T}t\right)x(t)\right),\qquad x\in\lm.$$
Thus $x_\beta$ is a critical point of $A_\beta$ restricted to $D$. By Lemma~\ref{lemma: admissible critical point}, $x_\beta$ is also a critical point of $A_\beta$ restricted to $\lm^H$.

Let
$$J=\begin{pmatrix}0&-1\\1&0\end{pmatrix}.$$
Then, for $x\in\lm$,
$$A_\beta(x)=\int_0^T\left[\frac12\sum_{i=1}^n m_i\left|\dot x_i+\frac{\beta}{T}Jx_i\right|^2+U(x)\right]dt.$$

For $h\in H$, set
$$\epsilon_h:=\det(\tau(h)),\qquad \delta_h:=\det(\rho(h)),\qquad \chi_h:=\epsilon_h\delta_h.$$
Since
$$J\rho(h)=\delta_h\rho(h)J,$$
for an arbitrary $x\in\lm$, if $z=h\cdot x$, $s=\tau(h^{-1})t$, and $j=\sigma(h^{-1})(i)$, then
$$\dot z_i(t)+\frac{\beta}{T}Jz_i(t)=\epsilon_h\rho(h)\left(\dot x_j(s)+\chi_h\frac{\beta}{T}Jx_j(s)\right).$$
Hence, by orthogonality, preservation of the masses, and invariance of the potential and the time-circle measure,
\begin{equation}\label{eq:A_beta transformation}
A_\beta(h\cdot x)=A_{\chi_h\beta}(x).
\end{equation}

Suppose first that $H$ is of type~R. Then $\chi_h=1$ for every $h\in H$, and \eqref{eq:A_beta transformation} shows that $A_\beta$ is $H$-invariant. Since $x_\beta$ is a critical point of $A_\beta$ restricted to $\lm^H$, Palais's principle of symmetric criticality \cite{palais1979principle} implies that $x_\beta$ is a critical point of $A_\beta$ on $\lm$. Equivalently, $y_\beta$ is a critical point of $A_T$ on $\lm_\beta$.

Since $y_\beta$ is collision-free, it satisfies the $n$-body equations on $(0,T)$. Moreover, the first variation with respect to twisted variations satisfying $\eta(T)=\exp(i\beta)\eta(0)$ gives
$$
\dot y_\beta(T)=\exp(i\beta)\dot y_\beta(0).
$$
Hence
$$
y_\beta(t+kT)=\exp(ik\beta)y_\beta(t)
$$
defines a solution on $\mathbb R$.

Now suppose that $x_\beta$ is constant, say $x_\beta(t)\equiv a$. Since $x_\beta\in D$, we have $a\in V^H$, and $a$ is collision-free. Since $x_\beta$ is a critical point of $A_\beta$ restricted to $\lm^H$, for every constant variation $v=(v_1,\ldots,v_n)\in V^H$,
\begin{equation}\label{eq:variation on constant x}
\sum_{i=1}^n\left\langle\nabla_iU(a)+m_i\frac{\beta^2}{T^2}a_i,v_i\right\rangle=0.
\end{equation}
Since $a\in V^H$, a direct computation using orthogonality of $\rho$ and preservation of the masses by $\sigma$ gives, for every $h\in H$,
$$\frac{1}{m_i}\nabla_iU(a)=\rho(h)\frac{1}{m_{\sigma(h^{-1})(i)}}\nabla_{\sigma(h^{-1})(i)}U(a),$$
and thus
$$\left(\frac{1}{m_i}\nabla_iU(a)+\frac{\beta^2}{T^2}a_i\right)_{i=1}^n$$
is fixed by the $H$-action. Moreover, from the form of the potential,
$$\sum_{i=1}^n\nabla_iU(z)=0$$
for every collision-free configuration $z\in V$. Since $a\in V^H\subset V$, it follows that
$$\sum_{i=1}^n m_i\left(\frac{1}{m_i}\nabla_iU(a)+\frac{\beta^2}{T^2}a_i\right)=0.$$
Therefore,
$$\left(\frac{1}{m_i}\nabla_iU(a)+\frac{\beta^2}{T^2}a_i\right)_{i=1}^n\in V^H.$$
Taking
$$v_i=\frac{1}{m_i}\nabla_iU(a)+\frac{\beta^2}{T^2}a_i$$
in \eqref{eq:variation on constant x}, we obtain
$$\nabla_iU(a)+m_i\frac{\beta^2}{T^2}a_i=0,\qquad i=1,\ldots,n.$$
Thus
$$y_\beta(t)=\exp\left(i\frac{\beta}{T}t\right)a$$
is a relative equilibrium. In particular, it satisfies the $n$-body equations on $[0,T]$, and its extension is a solution on $\mathbb R$.

It remains to prove the converse. Suppose that $y_\beta$ satisfies the $n$-body equations on $[0,T]$ and that $H$ is not of type~R. Then $x_\beta$ satisfies the Euler--Lagrange equation for $A_\beta$ on $(0,T)$. In particular,
$$dA_\beta(x_\beta)[\eta]=0$$
for every variation $\eta$ compactly supported in $(0,T)$.

Since $H$ is not of type~R, there exists $h\in H$ such that $\chi_h=-1$. By \eqref{eq:A_beta transformation},
$$A_\beta(h\cdot x)=A_{-\beta}(x).$$
Let $t_h\in[0,T)$ represent $\tau(h^{-1})(0)$ on the time circle. For every variation $\eta$ compactly supported in $(0,T)\setminus\{t_h\}$, the variation $h\cdot\eta$ is compactly supported in $(0,T)$. Since $h\cdot x_\beta=x_\beta$ and $x_\beta$ satisfies the Euler--Lagrange equation for $A_\beta$ on $(0,T)$, differentiating the above identity at $x_\beta$ gives
$$dA_{-\beta}(x_\beta)[\eta]=dA_\beta(x_\beta)[h\cdot\eta]=0.$$
Hence $x_\beta$ also satisfies the Euler--Lagrange equation for $A_{-\beta}$ on $(0,T)\setminus\{t_h\}$.

On this set, the Euler--Lagrange equations for $A_\beta$ and $A_{-\beta}$ are
$$m_i\left(\ddot x_{\beta,i}+2\frac{\beta}{T}J\dot x_{\beta,i}-\frac{\beta^2}{T^2}x_{\beta,i}\right)=\nabla_iU(x_\beta)$$
and
$$m_i\left(\ddot x_{\beta,i}-2\frac{\beta}{T}J\dot x_{\beta,i}-\frac{\beta^2}{T^2}x_{\beta,i}\right)=\nabla_iU(x_\beta).$$
Subtracting the two equations gives
$$4m_i\frac{\beta}{T}J\dot x_{\beta,i}=0.$$
Since $\beta\neq0$, it follows that
$$\dot x_{\beta,i}=0\quad\text{almost everywhere on }(0,T),\qquad i=1,\ldots,n.$$
Thus $x_\beta$ is constant on $[0,T]$.

This proves the equivalence.
\end{proof}

\begin{remark}
If $V^H=\{0\}$, the possibility that $x_\beta$ is constant is excluded by the collision-free assumption. Hence, in this case, $y_\beta$ satisfies the $n$-body equations on $[0,T]$, and equivalently its extension is a solution on $\mathbb R$, if and only if $H$ is of type~R.
\end{remark}

We now combine the preceding propositions to prove the main theorems.

\begin{proof}[Proof of Theorem~\ref{theorem: weak force, glocal minimizer version}]
By Proposition~\ref{prop: uniform coercivity} and the inclusion $D_\beta\subset(\lm^H)_\beta$, any action-minimizing sequence in $D_\beta$ has bounded $H^1$-norm. Since $H^1([0,T],V)$ embeds continuously into $C^0([0,T],V)$, the sequence is also uniformly bounded in the $L^\infty$-norm. Since $D$ is $H$-admissible, it is weakly closed in $\lm^H$. Therefore, by Lemma~\ref{lemma:direct method compactness}, the standard compactness argument shows that $A_T$ attains its minimum on $D_\beta$ for every $|\beta|<2\pi$.

Let $y_\beta$ be an arbitrary minimizer of $A_T$ over $D_\beta$, and let $x_0\in\mathcal M_0$. Proposition~\ref{prop: estimate of action after applying rotation} gives
\begin{equation}\label{eq:upper estimate on action of y_beta}
A_T(y_\beta)\leq A_T\left(\exp\left(i\frac{\beta}{T}t\right)x_0\right)\leq A_T(x_0)+\delta\left(0,\beta,A_T(x_0)\right).
\end{equation}
It follows that $A_T(y_\beta)$, and hence the kinetic part of the action of $y_\beta$, is uniformly bounded by some $K_1>0$ for all $|\beta|<2\pi$.

Since
$$
\exp\left(-i\frac{\beta}{T}t\right)y_\beta\in D,
$$
the minimizing property of $x_0$ gives
\begin{equation}\label{eq:lower estimate on action of y_beta}
A_T(x_0)\leq A_T\left(\exp\left(-i\frac{\beta}{T}t\right)y_\beta\right)\leq A_T(y_\beta)+\delta(\beta,-\beta,K_1).
\end{equation}
Combining \eqref{eq:upper estimate on action of y_beta}, \eqref{eq:lower estimate on action of y_beta}, and the explicit formula for $\delta$ in \eqref{eq: explicit value of delta}, we obtain
$$
A_T(y_\beta)\longrightarrow A_T(x_0)\qquad\text{as }\beta\longrightarrow0.
$$

Let
$$
m_{\mathrm{col}}:=\inf\left\{A_T(c):c\in D\text{ is colliding}\right\},
$$
with the convention that the infimum of the empty set is $+\infty$. By the same compactness argument applied to an action-minimizing sequence of colliding curves, together with the assumption that every element of $\mathcal M_0$ is collision-free, there exists $\epsilon>0$ such that
$$
A_T(x_0)+\epsilon\leq m_{\mathrm{col}}.
$$
The preceding estimates are uniform in the choice of the minimizer $y_\beta$. Thus we may choose $\theta>0$ sufficiently small that, for every $|\beta|<\theta$,
$$
\left|A_T(y_\beta)-A_T(x_0)\right|+\delta(\beta,-\beta,K_1)<\epsilon.
$$
Then
\begin{equation}\label{ineq: estimate on rotated back y_beta}
\begin{split}
A_T\left(\exp\left(-i\frac{\beta}{T}t\right)y_\beta\right)
&\leq A_T(y_\beta)+\delta(\beta,-\beta,K_1)\\
&\leq A_T(x_0)+\left|A_T(y_\beta)-A_T(x_0)\right|+\delta(\beta,-\beta,K_1)\\
&<A_T(x_0)+\epsilon\\
&\leq m_{\mathrm{col}}.
\end{split}
\end{equation}
Hence
$$
\exp\left(-i\frac{\beta}{T}t\right)y_\beta
$$
cannot have a collision. Since multiplication by $\exp\left(i\beta t/T\right)$ preserves collisions, $y_\beta$ is collision-free for every $|\beta|<\theta$.

Moreover, given any sequence $\beta_k\rightarrow0$ and $y_{\beta_k}\in\mathcal M_{\beta_k}$, by Lemma~\ref{lemma:direct method compactness}, after passing to a subsequence, still denoted by $\beta_k$, we may assume that $y_{\beta_k}$ converges uniformly to some $y^*$ and
$$
A_T(y^*)\leq\liminf_{k\to\infty}A_T(y_{\beta_k})=A_T(x_0).
$$
Since $\exp\left(-i\frac{\beta_k}{T}t\right)y_{\beta_k}\in D$ has uniformly bounded $H^1$-norm and $D$ is weakly closed in $\lm^H$, after passing to a subsequence, it converges weakly to some $z\in D$. On the other hand, the rotated-back sequence converges uniformly to $y^*$, which forces $z=y^*\in D$. Hence $y^*\in\mathcal M_0$. Therefore,
$$
A_T(y_{\beta_k})\longrightarrow A_T(y^*).
$$
Applying the second statement of Lemma~\ref{lemma:direct method compactness}, we obtain
$$
y_{\beta_k}\longrightarrow y^*\qquad\text{strongly in }H^1.
$$
Consequently,
$$
\sup_{y\in\mathcal M_\beta}\operatorname{dist}_{H^1}(y,\mathcal M_0)\longrightarrow0\qquad\text{as }\beta\longrightarrow0.
$$
Lastly, by Proposition~\ref{prop: collision-free minimizer is solution} for $\beta\neq0$, and by Lemma~\ref{lemma: admissible critical point} together with Palais's principle for $\beta=0$, every $y_\beta$ with $|\beta|<\theta$ extends, via
$$
y_\beta(t+kT)=\exp(ik\beta)y_\beta(t),
$$
to a solution of the $n$-body equations on $\mathbb R$.

\end{proof}

In many applications, collisions in $D$ are excluded by obtaining
an explicit lower bound $\mathcal A_*$ for the action of colliding paths and comparing
it with the action of an explicitly constructed test path
$x_{\mathrm{test}}\in D$. In such cases, one can also obtain an
explicit admissible value of $\theta$ in
Theorem~\ref{theorem: weak force, glocal minimizer version}.

\begin{corollary}
\label{corollary: explicit value of theta in glocal theorem}
Assume the hypotheses of
Theorem~\ref{theorem: weak force, glocal minimizer version}.
Suppose additionally that there exist a constant $\mathcal A_*$ and a
test path $x_{\mathrm{test}}\in D$ such that
$$\mathcal A_{\mathrm{test}}:=A_T(x_{\mathrm{test}})<\mathcal A_*,$$
and such that, for every $y_\beta\in\mathcal M_\beta$ with
$|\beta|< 2\pi$, if $y_\beta$ has a collision, then
$$A_T\left(\exp\left(-i\frac{\beta}{T}t\right)y_\beta\right)\geq\mathcal A_*.$$
Then one may choose
\begin{equation}\label{eq:explcit value of theta}
\theta
=
2\pi\frac{\sqrt{\mathcal A_*}-\sqrt{\mathcal A_{\mathrm{test}}}}
{\sqrt{\mathcal A_*}+\sqrt{\mathcal A_{\mathrm{test}}}}
.
\end{equation}
\end{corollary}

\begin{proof}
For $|\beta|<2\pi$ and $y_\beta\in\mathcal M_\beta$, repeating the computation in
\eqref{eq:upper estimate on action of y_beta} and using
Proposition~\ref{prop: estimate of action after applying rotation}, we obtain
\begin{equation}\label{eq:upper estimate on action of y_beta, corollary}
\begin{split}
A_T(y_\beta)
&\leq
A_T\left(\exp\left(i\frac{\beta}{T}t\right)x_{\mathrm{test}}\right)\\
&\leq
\mathcal A_{\mathrm{test}}
+
\delta\left(0,\beta,\mathcal A_{\mathrm{test}}\right).
\end{split}
\end{equation}
For simplicity, set
$$B_\beta:=\mathcal A_{\mathrm{test}}+\delta\left(0,\beta,\mathcal A_{\mathrm{test}}\right).$$
Then
$$A_T(y_\beta)\leq B_\beta.$$
Using the explicit expression for $\delta$ in \eqref{eq: explicit value of delta}, we have
\begin{equation}\label{eq: explicit B beta}
\begin{split}
B_\beta
&=
\mathcal A_{\mathrm{test}}
\left(
1+\frac{|\beta|}{\pi}+\frac{\beta^2}{4\pi^2}
\right)\\
&=
\mathcal A_{\mathrm{test}}
\left(
1+\frac{|\beta|}{2\pi}
\right)^2.
\end{split}
\end{equation}
In order to exclude collisions, it is enough to show that
$$A_T\left(\exp\left(-i\frac{\beta}{T}t\right)y_\beta\right)<\mathcal A_*.$$
Since the kinetic part of $y_\beta$ is bounded by $B_\beta$,
Proposition~\ref{prop: estimate of action after applying rotation} gives
\begin{equation}
\begin{split}
A_T\left(\exp\left(-i\frac{\beta}{T}t\right)y_\beta\right)
&\leq
A_T(y_\beta)+\delta\left(\beta,-\beta,B_\beta\right)\\
&\leq
B_\beta+\delta\left(\beta,-\beta,B_\beta\right).
\end{split}
\end{equation}
Using the explicit expression for $\delta$ in \eqref{eq: explicit value of delta} and \eqref{eq: explicit B beta}, we have
\begin{equation}
\begin{split}
B_\beta+\delta\left(\beta,-\beta,B_\beta\right)
&=
B_\beta
\left(
1+\frac{2|\beta|}{2\pi-|\beta|}
+\frac{\beta^2}{(2\pi-|\beta|)^2}
\right)\\
&=
B_\beta
\left(
\frac{2\pi}{2\pi-|\beta|}
\right)^2\\
&=
\mathcal A_{\mathrm{test}}
\left(
\frac{2\pi+|\beta|}{2\pi-|\beta|}
\right)^2.
\end{split}
\end{equation}
Thus it suffices that
$$\mathcal A_{\mathrm{test}}\left(\frac{2\pi+|\beta|}{2\pi-|\beta|}\right)^2<\mathcal A_*.$$
Since $|\beta|<2\pi$, both sides are positive, and the preceding inequality is equivalent to
$$\sqrt{\mathcal A_{\mathrm{test}}}\left(2\pi+|\beta|\right)
<
\sqrt{\mathcal A_*}\left(2\pi-|\beta|\right).$$
Rearranging gives
$$|\beta|
<
2\pi\frac{\sqrt{\mathcal A_*}-\sqrt{\mathcal A_{\mathrm{test}}}}
{\sqrt{\mathcal A_*}+\sqrt{\mathcal A_{\mathrm{test}}}}.$$
Therefore one may choose
$$\theta
=
2\pi\frac{\sqrt{\mathcal A_*}-\sqrt{\mathcal A_{\mathrm{test}}}}
{\sqrt{\mathcal A_*}+\sqrt{\mathcal A_{\mathrm{test}}}}.$$
For every $|\beta|<\theta$,
$$
A_T\left(\exp\left(-i\frac{\beta}{T}t\right)y_\beta\right)<\mathcal A_*,
$$
so by the hypothesis of the corollary, $y_\beta$ is collision-free.
\end{proof}

    \begin{proof}[Proof of Theorem~\ref{theorem: weak force, local minimizer version}]
Since $H^1([0,T],V)$ embeds continuously into $C^0([0,T],V)$, $x_0$ is collision-free, and $D$ is $H$-admissible, we may choose $r>0$ so small that the closed $H^1$-ball of radius $r$ about $x_0$ in $\lm^H$ is contained in $D$ and uniformly separated from collisions. Decreasing $r$ if necessary, the strict local minimality of $x_0$ gives
$$A_T(x)>A_T(x_0)$$
for every $x\in\lm^H$ with
$$0<\norm{x-x_0}_{H^1}\leq r.$$
In fact, there exists $\epsilon>0$ such that
$$A_T(x_0)+\epsilon\leq\inf_{\substack{x\in\lm^H\\ \norm{x-x_0}_{H^1}=r}}A_T(x).$$
Otherwise we could find a sequence $(x^{[k]})$ on the $H^1$-sphere of radius $r$ such that
$$A_T(x^{[k]})\longrightarrow A_T(x_0).$$
By Lemma~\ref{lemma:direct method compactness}, after passing to a subsequence, $x^{[k]}$ converges uniformly and weakly in $H^1$ to some $x^*\in D$. Lower semicontinuity and strict local minimality give $x^*=x_0$, and the second statement of Lemma~\ref{lemma:direct method compactness} then gives $x^{[k]}\to x_0$ strongly in $H^1$, a contradiction.

Set
$$B_r:=\left\{x\in\lm^H:\norm{x-x_0}_{H^1}\leq r\right\}.$$
By the choice of $r$, we have $B_r\subset D$, and every curve in $B_r$ is collision-free. Since $B_r$ is bounded in $H^1$, there exist $K_1,K_2>0$ such that, for every $x\in B_r$,
$$\int_0^T K(x(t),\dot x(t))\,dt\leq K_1,\qquad \norm{x}_\infty\leq K_2.$$

Note that the estimate for $|A_\beta(x)-A_T(x)|$ corresponds to the special case $\beta_1=0$ and $\beta_2=\beta$ of \eqref{eq:estimate of action difference after rotation}. This particular estimate remains valid without assuming $V^H=\{0\}$. Applying the Cauchy--Schwarz inequality to the first term on the right-hand side of \eqref{eq:estimate of action difference after rotation}, we obtain
$$|A_\beta(x)-A_T(x)|\leq|\beta|K_2\sqrt{\frac{2m_{\max}K_1}{T}}+\frac{\beta^2m_{\max}K_2^2}{2T}.$$
Set
$$\delta_{B_r}(\beta):=|\beta|K_2\sqrt{\frac{2m_{\max}K_1}{T}}+\frac{\beta^2m_{\max}K_2^2}{2T}.$$
Then
$$|A_\beta(x)-A_T(x)|\leq\delta_{B_r}(\beta)$$
uniformly for $x\in B_r$, and
$$\delta_{B_r}(\beta)\longrightarrow0\qquad\text{as }\beta\longrightarrow0.$$
Therefore, there exists $\theta>0$ such that, for every $|\beta|<\theta$,
$$2\delta_{B_r}(\beta)<\epsilon,$$
and hence
$$\inf_{\substack{x\in\lm^H\\ \norm{x-x_0}_{H^1}=r}}A_\beta(x)>A_\beta(x_0).$$

By the standard compactness argument, $A_\beta$ attains its minimum over $B_r$ at some $x_\beta$. The above inequality implies that $x_\beta$ lies in the interior of $B_r$ for every $|\beta|<\theta$, and hence is a local minimizer of $A_\beta$ over $\lm^H$, and therefore also over $D$. Consequently,
$$y_\beta(t):=\exp\left(i\frac{\beta}{T}t\right)x_\beta(t)$$
is a local minimizer of $A_T$ over $D_\beta$. By the choice of $r$, $x_\beta$ is collision-free, and therefore $y_\beta$ is also collision-free.

We now show that $y_\beta$ converges to $x_0$ in $H^1$. Note that
\begin{equation*}
\begin{split}
A_T(x_0)
&\leq A_T(x_\beta)\\
&\leq A_\beta(x_\beta)+\delta_{B_r}(\beta)\\
&\leq A_\beta(x_0)+\delta_{B_r}(\beta)\\
&\leq A_T(x_0)+2\delta_{B_r}(\beta).
\end{split}
\end{equation*}
Hence
$$A_T(x_\beta)\longrightarrow A_T(x_0)\qquad\text{as }\beta\longrightarrow0.$$
Given any sequence $\beta_k\to0$, by Lemma~\ref{lemma:direct method compactness} and the strict local minimality of $x_0$, after passing to a subsequence, still denoted by $\beta_k$, we have
$$x_{\beta_k}\longrightarrow x_0\qquad\text{strongly in }H^1.$$
Since $\beta_k\to0$, it follows that
$$y_{\beta_k}=\exp\left(i\frac{\beta_k}{T}t\right)x_{\beta_k}(t)\longrightarrow x_0\qquad\text{strongly in }H^1.$$
Thus
$$\norm{y_\beta-x_0}_{H^1}\longrightarrow0\qquad\text{as }\beta\longrightarrow0.$$

Lastly, by Proposition~\ref{prop: collision-free minimizer is solution} for $\beta\neq0$, and by Lemma~\ref{lemma: admissible critical point} together with Palais's principle for $\beta=0$, $y_\beta$ extends, via
$$y_\beta(t+kT)=\exp(ik\beta)y_\beta(t),$$
to a solution of the $n$-body equations on $\mathbb R$.
\end{proof}

\begin{remark}
The setting of the main theorems can also be generalized to
$(\mathbb R^3)^n$ after fixing a nonzero vector
$\omega\in\mathbb R^3$ and considering the one-parameter family of
rotations with angular velocities $\lambda\omega$, $\lambda\in\mathbb R$.
In this case, the symmetry representations take the form
$$\rho:H\to O(3),\qquad \tau:H\to O(2),\qquad \sigma:H\to S_n.$$
A natural analogue of the planar type~R condition is
$$\det(\tau(h))\det(\rho(h))\,\rho(h)\omega=\omega
\qquad\text{for every }h\in H.$$
The planar setting is recovered by taking
$$\omega=(0,0,1)$$
and identifying $O(2)$ with the subgroup of $O(3)$ acting on the
$xy$-plane and fixing the $z$-axis. In this case,
$$\rho(h)\omega=\omega,$$
and hence the above condition reduces to
$$\det(\rho(h))\det(\tau(h))=1,$$
which is precisely the planar type~R condition.

With this modified symmetry action and compatibility condition, the
arguments of Theorems~\ref{theorem: weak force, glocal minimizer version}
and~\ref{theorem: weak force, local minimizer version} extend to the
spatial setting. See \cite{chenciner2005rotating} for a related
treatment of rotating solutions in the spatial three-body problem,
based on the implicit-function theorem rather than variational methods.

Although both weak-force results extend formally to $\mathbb R^3$, it is
worth noting that the collision-free ordered configuration space in
dimension three is simply connected. Thus the nontrivial fundamental
group available in the planar problem no longer provides a topological
mechanism for constraining global minimizers. Consequently, the hypothesis in
Theorem~\ref{theorem: weak force, glocal minimizer version} requiring
all global minimizers to be collision-free may be substantially harder
to satisfy in the spatial setting. This issue does not enter
Theorem~\ref{theorem: weak force, local minimizer version} in the same
way, since local minimality is a local property.
\end{remark}

%% file: part5.tex
\subsection{The figure-eight orbit}

Consider the Newtonian three-body problem in the weak-force regime, namely
$$
\alpha=1,\qquad n=3,\qquad m_1=m_2=m_3=1,
$$
and let $x_8$ denote the figure-eight solution with period $T$.

We introduce three symmetry operators on $\lm$ and define
$$
H:=\langle S_C,-S_T,S_M\rangle,
$$
where, in terms of the representations $\rho,\tau,\sigma$ introduced
above, $S_C,S_T,S_M$ are given by
$$
\begin{array}{c|c|c|c}
& \rho & \tau & \sigma\\
\hline
S_C
& I
& t\mapsto t-T/3
& (123)\\
S_T
& I
& t\mapsto -t
& (23)\\
S_M
& \mu
& t\mapsto t+T/2
& \mathrm{id},
\end{array}
$$
with
$$
\mu=
\begin{pmatrix}
-1&0\\
0&1
\end{pmatrix}.
$$
The figure-eight satisfies
$$
S_Cx_8=x_8,\qquad
S_Tx_8=-x_8,\qquad
S_Mx_8=x_8,
$$
and hence
$$
x_8\in\lm^H.
$$
The Chenciner--Montgomery variational construction shows that $x_8$
is a global minimizer of the action under these symmetry constraints
\cite{chenciner2000remarkable}.

\begin{figure}[tbp]
\centering
\includegraphics[width=\textwidth]{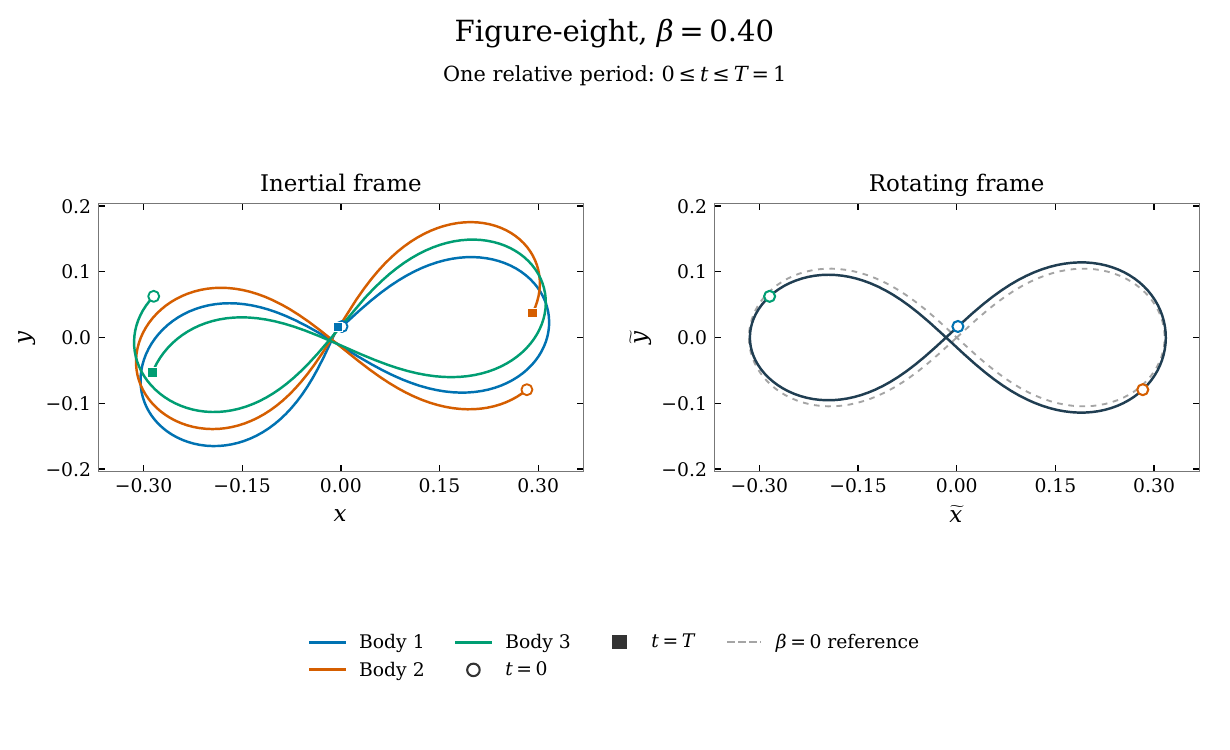}
\caption{Numerical continuation of the figure-eight in the $H'$-symmetry
class. Left: inertial trajectories
$y_i(t)$ over one relative period. Right: the periodic profile
$x_i(t)=e^{-i\beta t/T}y_i(t)$, with the shared path drawn once.
Hollow circles mark $t=0$, filled squares mark inertial endpoints,
and the dashed gray curve is the $\beta=0$ reference.}
\label{fig:rotating-figure-eight}
\end{figure}
The full group $H$ is not of type~R. Indeed,
$$
\det\rho(S_C)\det\tau(S_C)=1,
$$
while
$$
\det\rho(-S_T)\det\tau(-S_T)
=
\det\rho(S_M)\det\tau(S_M)
=
-1.
$$
Thus, in order to apply our theorem, we pass to the type~R subgroup
$$
H'
=
\langle S_C,-S_TS_M\rangle.
$$
In this symmetry class, $x_8$ is no longer a global minimizer (see
\cite{Barutello2008Symmetry}). We therefore investigate instead whether
$x_8$ is a strict local minimizer.

To this end, we use the numerical Hessian computation in a preprint by
Fujiwara, Fukuda, and Ozaki
\cite{fujiwara2018decompositionhessianmatrixaction}. Let $W$ denote the zero-center-of-mass variation space. For
$i\in\{C,T,M\}$, set
$$
W_i:=\{\Psi\in W:S_i\Psi=\Psi\}.
$$
The corresponding orthogonal projections commute, and hence decompose
$W$ into $2^3$ mutually orthogonal symmetry sectors. Moreover, for
$i\in\{T,M\}$,
$$
W_i^\perp=\{\Psi\in W:S_i\Psi=-\Psi\}.
$$
An $H'$-equivariant variation $\Psi$ satisfies
$$
S_C\Psi=\Psi,
\qquad
S_TS_M\Psi=-\Psi.
$$
Hence the $H'$-equivariant variation space is
$$
\left(W_C\cap W_T\cap W_M^\perp\right)
\oplus
\left(W_C\cap W_T^\perp\cap W_M\right),
$$
which corresponds precisely to the sectors denoted by $cc-$ and $cs+$,
respectively, in
\cite{fujiwara2018decompositionhessianmatrixaction}. Fujiwara,
Fukuda, and Ozaki numerically obtain the lowest normalized Hessian
eigenvalues in these two sectors:
$$
\widetilde\lambda_{cc-}=1.93391,
\qquad
\widetilde\lambda_{cs+}=2.32495.
$$
Thus their computation provides numerical evidence that the Hessian
of the action at $x_8$ is positive on the $H'$-equivariant variation
space, and hence numerical evidence that $x_8$ is a strict local
minimizer in $\lm^{H'}$.

Consequently, under this numerically supported strict-local-minimality assumption, Theorem~\ref{theorem: weak force, local minimizer version} applies. Hence, for all sufficiently small $|\beta|$, there exists a collision-free local minimizer
$$
y_\beta\in(\lm^{H'})_\beta
$$
which gives a solution of the three-body problem and satisfies
$$
\norm{y_\beta-x_8}_{H^1}\longrightarrow0
\qquad\text{as }\beta\longrightarrow0.
$$

As illustrated in Figure~\ref{fig:rotating-figure-eight}, in the rotating frame the right lobe of the figure-eight orbit is slightly larger than the left one. This asymmetry reflects the loss of the reflection symmetry across the $y$-axis when the symmetry group is reduced from $H$ to $H'$.

\subsection{Double choreographic loops with q-fold rotation symmetry}

Consider the planar Newtonian four-body problem with equal masses,
$$
m_1=m_2=m_3=m_4=1,
\qquad
\alpha=1.
$$
For $q=3,5$, let
$$
H_q:=\langle h_1,h_2,h_3\rangle
$$
act on $\lm$ by
$$
\begin{aligned}
(h_1\cdot x)(t)
&=
\bigl(
\overline{x}_2,\overline{x}_1,
\overline{x}_4,\overline{x}_3
\bigr)(-t),\\
(h_2\cdot x)(t)
&=
e^{\pi i/q}
\bigl(
x_3,x_4,x_1,x_2
\bigr)
\left(t+\frac{T}{2q}\right),\\
(h_3\cdot x)(t)
&=
-\bigl(
x_2,x_1,x_4,x_3
\bigr)(t).
\end{aligned}
$$
The corresponding fixed-point space is
$$
\lm^{H_q}
=
\left\{
x\in\lm:
h_1\cdot x=x,\;
h_2\cdot x=x,\;
h_3\cdot x=x
\right\}.
$$
It is clear that $V^{H_q}=\{0\}$ and that $H_q$ is of type~R.
For $q=3,5,$ Chen proved that every minimizer of $A_T$ on $\lm^{H_q}$ is
collision-free \cite{chen2003binary}. Hence
Theorem~\ref{theorem: weak force, glocal minimizer version} applies,
yielding collision-free $\beta$-periodic solutions for all
$|\beta|<\theta$ for some $\theta>0$. 

\begin{figure}[!t]
    \centering

    \begin{subfigure}{0.8\textwidth}
        \centering
        \includegraphics[width=\linewidth]
            {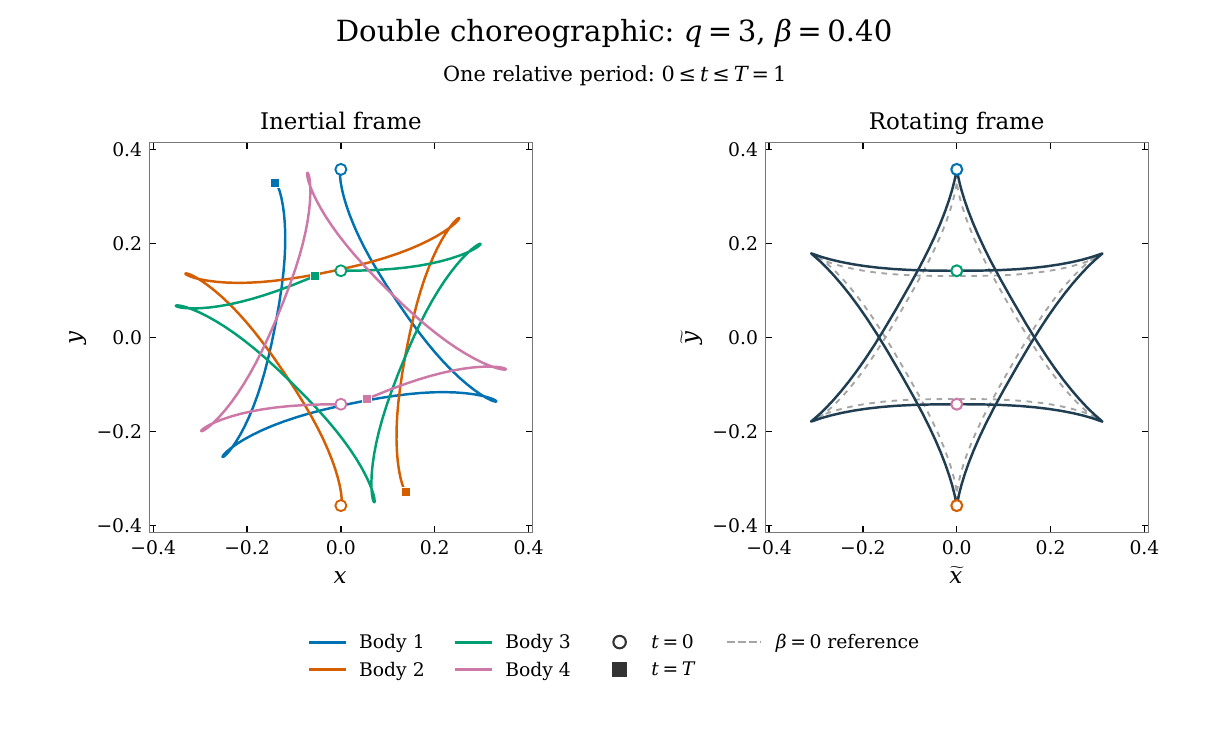}
        \label{fig:double-q3}
    \end{subfigure}

    \caption{Numerical continuation of double choreographic orbits at $\beta=0.40$, outside the explicit sufficient ranges obtained below. The trajectories are displayed using the same conventions as in
Figure~\ref{fig:rotating-figure-eight}.}
    \label{fig:double-choreographic}
\end{figure}

Moreover, with $T=1$, Chen obtained explicit lower bounds for the action
of paths having a collision at one of the boundary points
$$t=0\qquad\text{or}\qquad t=\frac{1}{4q},$$
and constructed test paths $x_{\mathrm{test}}\in\lm^{H_q}$ whose
actions are strictly below the corresponding bounds. These estimates
are given in the proof of Theorem~5 of his paper for $q=3$ and in the
subsequent discussion for $q=5$.

Therefore, Corollary~\ref{corollary: explicit value of theta in glocal theorem}
gives an explicit admissible value of $\theta$ in the case of $T=1$, provided that we show
that any colliding minimizer $y_\beta$ must also have a collision at
$t=0$ or $t=\frac{1}{4q}$.

Let $y_\beta$ be a minimizer of $A_T$ over $(\lm^{H_q})_\beta$ for
$|\beta|<2\pi$, and set
$$x_\beta(t):=e^{-i\beta t/T}y_\beta(t).$$
Then $x_\beta$ is a minimizer of $A_\beta$ over $\lm^{H_q}$. Let
$$T_*:=\frac{T}{4q},\qquad I:=[0,T_*],$$
be a fundamental domain for the action of $H_q$ on the time circle. Here define
$$K:=\ker\tau=\langle h_3\rangle.$$
Since $\rho(h_3)=-I$ is a rotation by $\pi$ and $h_3$ exchanges the
bodies in the pairs $(1,2)$ and $(3,4)$, $K$ has the rotating circle
property introduced by Ferrario--Terracini
\cite{ferrario2004existence}. Hence, by their Theorem~10.7,
$x_\beta$ has no interior collisions on $I$. By $H_q$-equivariance, any
collision has a symmetry image in $I$, and therefore any colliding
$y_\beta$ must also have a collision on the boundary of $I$, namely at
$t=0$ or $t=T_*$. Ferrario--Terracini also verify the rotating circle property
for $K$ in their Example~11.4.

We may now apply Corollary~\ref{corollary: explicit value of theta in glocal theorem}. Here $\mathcal A_{\mathrm{test}}$ denotes the action of Chen's test path, and $\mathcal A_*$ denotes the corresponding boundary-collision action lower bound. For $q=3$,
$$\mathcal A_{\mathrm{test}}\approx 26.0476,\qquad \mathcal A_*\approx26.2386,$$
and hence
$$\theta
=
2\pi\frac{\sqrt{\mathcal A_*}-\sqrt{\mathcal A_{\mathrm{test}}}}
{\sqrt{\mathcal A_*}+\sqrt{\mathcal A_{\mathrm{test}}}}
\approx 0.01148.$$
Similarly, for $q=5$,
$$\mathcal A_{\mathrm{test}}\approx 33.8819,\qquad \mathcal A_*\approx 36.2246,$$
which gives
$$\theta
=
2\pi\frac{\sqrt{\mathcal A_*}-\sqrt{\mathcal A_{\mathrm{test}}}}
{\sqrt{\mathcal A_*}+\sqrt{\mathcal A_{\mathrm{test}}}}
\approx 0.1050.$$

It is worth noting that this application overlaps with the earlier
construction of Chen \cite{CHEN_2003}. Indeed, it is straightforward to
verify that, with
$$
S=\frac{T}{q},
\qquad
d=\frac{q}{1-\beta/(2\pi)},
$$
and with $G$ the four-body symmetry group introduced in Section~5 of
\cite{CHEN_2003}, Chen's space $H_{d,S}^G$ coincides with
$(\lm^{H_q})_\beta$. Chen proves that, for $d\geq d_0$ with
$d_0\leq3.55$, every global minimizer in $H_{d,S}^G$ is collision-free.
Consequently, his result already covers the rotating family obtained
here from the case $q=5$ for all sufficiently small $|\beta|$.

For $q=3$, Chen also verifies the case $d=3$ by a numerical
test-path comparison and notes that the bound on $d_0$ can be improved
using numerical methods. This suggests that part of the rotating family
near $q=3$ may also be covered by his approach, although such a
neighborhood of $d=3$ is not established there.

\subsection{Counter-rotating double choreographies in planar $(N+3)$-body problems}

Consider the planar Newtonian $(N+3)$-body problem with equal masses,
$$m_1=\cdots=m_{N+3}=1,\qquad \alpha=1.$$
Given a parameter choice
$$\mathbf p:=(N,r,k_1,k_2,d),$$
let
$$H_{\mathbf p}:=\langle h_1,h_2,h_3\rangle$$
act on $\lm$ by
\begin{align*}
(h_1\cdot x)(t)
&=
\exp\left(-\frac{2\pi i d}{r}\right)
\bigl(x_1,\ldots,x_{N+3}\bigr)
\left(t+\frac{T}{r}\right),\\
(h_2\cdot x)(t)
&=
\bigl(
x_1,\ldots,x_N,
x_{N+3},x_{N+1},x_{N+2}
\bigr)
\left(t+\frac{T}{3}\right),\\
(h_3\cdot x)(t)
&=
\bigl(
x_N,x_1,\ldots,x_{N-1},
x_{N+1},x_{N+2},x_{N+3}
\bigr)
\left(t+\frac{T}{N}\right).
\end{align*}
Define
$$
D_{\mathbf p}
:=
\overline{
\left\{
x\in\lm^{H_{\mathbf p}}:
\begin{array}{l}
x_i(t)\neq x_j(t)\text{ for all }i\neq j\text{ and }t\in[0,T],\\
\deg(x_i-x_j)=k_1,\qquad 1\leq i<j\leq N,\\
\deg(x_i-x_j)=k_2,\qquad N+1\leq i<j\leq N+3
\end{array}
\right\}
},
$$
where $\deg(x_i-x_j)$ denotes the winding number of the relative loop $x_i-x_j$ about the origin. This is precisely the closure $\overline{\Lambda}_2$ considered by Yuan and Zhang \cite{yuan2018new}. They prove that for
$$
\mathbf p=(4,7,3,-4,3),\qquad
(5,8,3,-5,3),\qquad
(7,10,3,-7,3),
$$
every global minimizer of $A_T$ over $D_{\mathbf p}$ is collision-free and hence gives a periodic solution of the corresponding $(N+3)$-body problem.
\begin{figure}[!t]
    \centering
    \begin{subfigure}{0.8\textwidth}
        \centering
        \includegraphics[width=\linewidth]
            {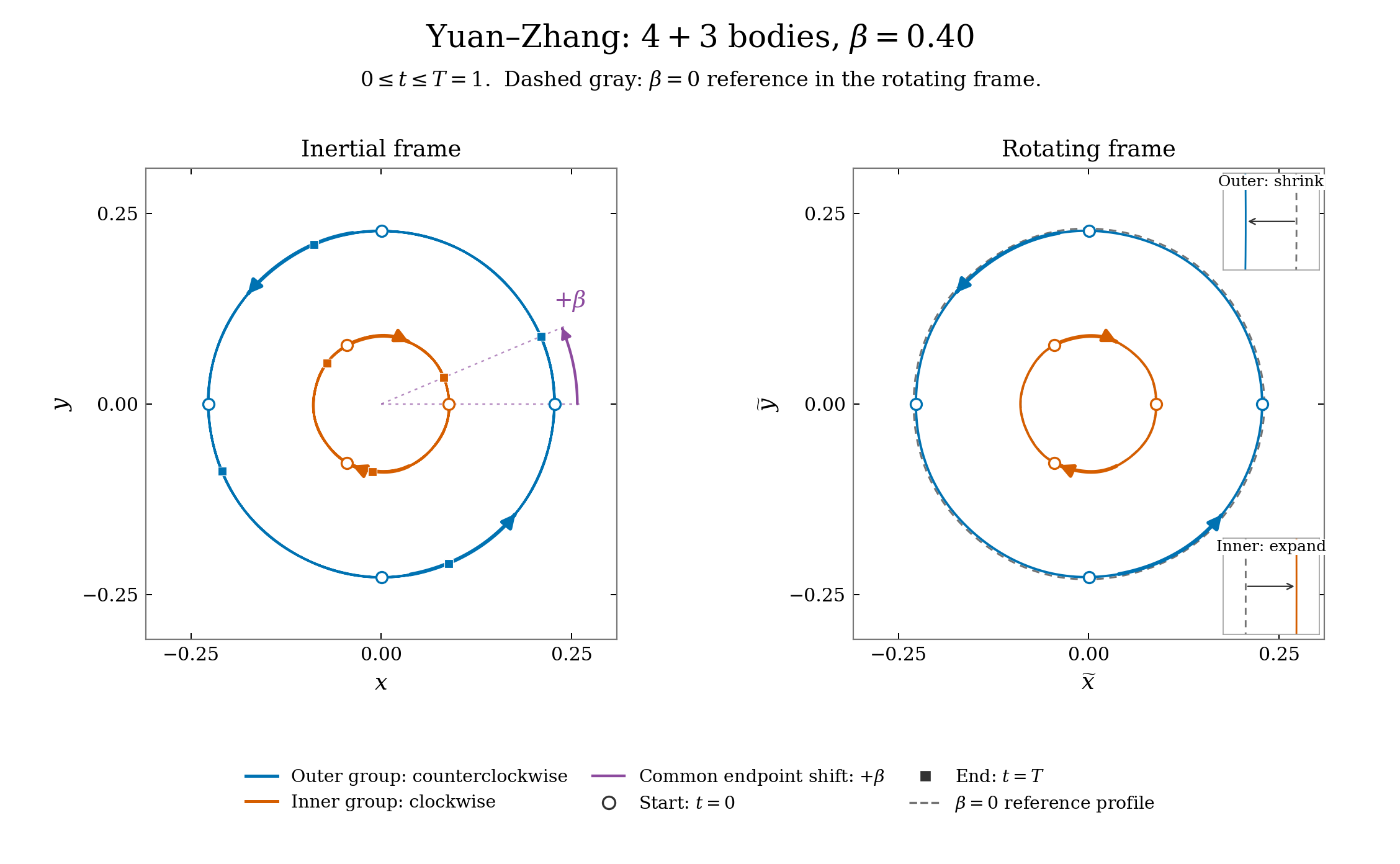}
    \end{subfigure}

    \caption{Numerical continuation of the counter-rotating double
choreography in the seven-body case ($N=4$) at $\beta=0.40$. The trajectories are displayed using the same conventions as in
Figure~\ref{fig:rotating-figure-eight}.}
    \label{fig:double-n7}
\end{figure}
For the above three choices of $\mathbf p$, $V^{H_{\mathbf p}}=\{0\}$ and $H_{\mathbf p}$ is of type~R. They also observe that $D_{\mathbf p}$ is weakly closed in $\lm^{H_{\mathbf p}}$. Moreover, since $H^1([0,T],V)$ embeds continuously into $C^0([0,T],V)$ and winding numbers are locally constant on collision-free loops, every collision-free element of $D_{\mathbf p}$ is an interior point of $D_{\mathbf p}$ relative to $\lm^{H_{\mathbf p}}$. Hence $D_{\mathbf p}$ is $H_{\mathbf p}$-admissible.

Therefore, Theorem~\ref{theorem: weak force, glocal minimizer version} applies, and for all sufficiently small $|\beta|$ every global minimizer
$$
y_\beta\in(D_{\mathbf p})_\beta
$$
is collision-free and extends to a solution of the $(N+3)$-body equations on $\mathbb R$.

Figure~\ref{fig:double-n7} illustrates how the orbit changes in the rotating frame, as highlighted in the two zoomed-in panels. The numerical computations suggest that, as $\beta$ increases, the inner circular component slightly expands while the outer component contracts.